\documentclass[11pt,a4paper,reqno]{amsart}

\usepackage{epsfig}
\usepackage{amsfonts}
\usepackage{amsmath}
\usepackage{amssymb}
\usepackage{amsthm}
\usepackage{times}
\usepackage{color}
\usepackage{enumerate}
\usepackage{graphicx}
\usepackage{nicefrac}
\usepackage[includehead,includefoot,margin=25mm]{geometry}
\usepackage[all]{xy}
\usepackage{multicol}

\newtheorem{theorem}{Theorem}[section]
\newtheorem*{theorem*}{Theorem}
\newtheorem{corollary}[theorem]{Corollary}
\newtheorem{lemma}[theorem]{Lemma}
\newtheorem{proposition}[theorem]{Proposition}

{\theoremstyle{remark}
  }
{\theoremstyle{definition}
  \newtheorem{definition}[theorem]{Definition}

  \newtheorem{example}[theorem]{Example}

}

\def\Z{\mathbb{Z}}

\newcommand{\ZZ}[0]{\ensuremath{\mathbb{Z}}}
\newcommand{\GA}[0]{\ensuremath{\mathbb{G}_{\mathrm{a}}}}
\newcommand{\GM}[0]{\ensuremath{\mathbb{G}_{\mathrm{m}}}}
\newcommand{\AF}[0]{\ensuremath{\mathbb{A}}}

\newcommand{\OO}[0]{\ensuremath{\mathcal{O}}}

\newcommand{\dom}[0]{\ensuremath{\operatorname{dom}}}

\newcommand{\fract}[0]{\ensuremath{\operatorname{Frac}}}

\newcommand{\spec}[0]{\ensuremath{\operatorname{Spec}}}

\newcommand{\ord}[0]{\ensuremath{\operatorname{ord}}}

\makeatletter \newcommand*\bigcdot{\mathpalette\bigcdot@{.5}} \newcommand*\bigcdot@[2]{\mathbin{\vcenter{\hbox{\scalebox{#2}{$\m@th#1\bullet$}}}}} \makeatother

\begin{document}

\title{On rational multiplicative group actions}

\author{Luis Cid}
\address{Instituto de Matem\'atica y F\'\i sica, Universidad de Talca,
  Casilla 721, Talca, Chile.}%
\email{luis.cid@inst-mat.utalca.cl}

\author{Alvaro Liendo} %
\address{Instituto de Matem\'atica y F\'\i sica, Universidad de Talca,
  Casilla 721, Talca, Chile.}%
\email{aliendo@inst-mat.utalca.cl}

\date{\today}

\thanks{
 \mbox{\hspace{11pt}}{\it Key words}: Rational multiplicative group actions; rational semisimple derivations; locally
nilpotent derivations.\\
 \mbox{\hspace{11pt}}Both authors were partially supported by the grant 346300 for IMPAN from the Simons Foundation and the matching 2015-2019 Polish MNiSW fund. The first author was also partially supported by CONICYT-PFCHA/Doctorado Nacional/2018-folio 21181058. The first and second authors were partially supported by Fondecyt project 1200502.}

\begin{abstract}
We establish a one-to-one correspondence between rational $\GM$-actions on an algebraic variety $X$ and derivations $\partial\colon K_X\to K_X$ of the field of fractions $K_X$ of $X$ satisfying that there exists a generating set $\{a_i\}_{i\in I}$ of $K_X$ as a field such that $\partial(a_i)=\lambda_i a_i$ with $\lambda_i \in \ZZ$ for all $i\in I$. We call such derivations rational semisimple. Furthermore, we also prove the existence of a rational slice for every rational semisimple derivation, i.e., an element $s\in K_X$ such that $\partial(s)=s$. By analogy with the case of additive group actions case, we prove that  $K_X\simeq K_X^{\GM}(s)$ and that under this isomorphism the derivation $\partial$ is given by $\partial=s\frac{d}{ds}$. Here, $K_X^{\GM}$ is the field of  invariant of the $\GM$-action.
\end{abstract}

\maketitle

\section{Introduction}

Let $k$ be an algebraically closed field of characteristic zero. By a variety we mean an integral separated scheme of finite type. We let $\OO_X$ be its structure sheaf and $K_X$  be its the field of rational functions so that $K_X=\fract \OO_X(U)$ for any affine open set $U\subset X$.

We also let $\GM$ and $\GA$ be the multiplicative group and the additive group over $k$, respectively. It is well known that regular additive group actions on affine varieties are in one-to-one correspondence with certain derivations called locally nilpotent \cite{Fre06,Dai03}. Indeed, letting $X$ be an affine variety a locally nilpotent derivation on $X$ is a $k$-derivation $\partial\colon \OO_X(X)\to \OO_X(X)$ such that for every $f\in \OO_X(X)$ there exists $i\in \ZZ_{\geq 0}$ with $\partial^i(f)=0$. All derivations in this paper are $k$-derivations so we will call them simply derivations. Given a locally nilpotent derivations $\partial$ on $X$ we obtain a regular $\GA$-actions $\varphi\colon \GA\times X\to X$  on $X$ via the exponential map 
$$\varphi^*=\exp(z\partial)\colon \OO_X(X)\to\OO_X(X)[z] \quad\mbox{given by}\quad f\mapsto \sum_{i\geq0}\frac{z^i\partial^i(f)}{i!}\,.$$
On the other hand, given a regular $\GA$-action $\varphi\colon \GA\times X\to X$  on $X$ we obtain a locally nilpotent derivations $\partial$ on $X$ via
$$\partial\colon \OO_X(X)\to \OO_X(X) \quad\mbox{given by}\quad  f\mapsto \operatorname{ev}_{0}\circ \frac{d}{dz}\circ \varphi^*(f)\,,$$
where $\operatorname{ev}_{0}\colon\OO_X(X)[z]\to \OO_X(X)$ is the evaluation morphism  in $z=0$.

In \cite{DuLi16} Dubouloz and the second author, introduced a class of rationally integrable derivations that generalized locally nilpotent derivations to the rational setting. Indeed, a derivation $\partial\colon K_X\to K_X$ is called rationally integrable if the exponential map 
$$\exp(z\partial)\colon K_X\to K_X[|z|] \quad\mbox{given by}\quad f\mapsto \sum_{i\geq0}\frac{z^i\partial^i(f)}{i!}\,,$$
factors through the ring $K_X(z)\cap K_X[|z|]$. Their main theorem provides a one-to-one correspondence between rationally integrable derivations $\partial\colon K_X\to K_X$ and rational $\GA$-actions $\varphi\colon \GA\times X\dashrightarrow X$ on $X$. The correspondence is given similarly to above via
$$\varphi^*=\exp(z\partial) \quad \mbox{and}\quad \partial= \operatorname{ev}_{0}\circ \frac{d}{dz}\circ \varphi^*(f)\,,$$
after recalling that $K_X(z)\cap K_X[|z|]=\{r(z)\in K_X(z)\mid \operatorname{ord}_0(r)\geq 0\}$ so that $\operatorname{ev}_{0}$ is well defined.

In this paper we expand and generalize the results in \cite{DuLi16} to allow a classification of rational $\GM$-action. It is well known that regular $\GM$-action on an affine variety $X$ are in one-to-one correspondence with semisimple derivations $\partial\colon \OO_X(X)\to \OO_X(X)$ having integer eigenvalues. Recall that such a derivation is semisimple if there exists a basis $\{a_i\}_{i\in I}$ of $A$ as vector space such that $\partial(a_i)=\lambda_i a_i$ with $\lambda_i\in k$. 

In Definition~\ref{def:rat-semisimple}, we introduce rational semisimple derivations. A derivation $\partial\colon K_X\to K_X$ on an algebraic variety $X$ is rational semisimple if there exist a generating set $\{a_i\}_{i\in I}$ of $K_X$ (as field) such that $\partial(a_i)=\lambda_i a_i$ with $\lambda_i\in \mathbb{Z}$. Every semisimple derivation whose eigenvalues are integers numbers is rational semisimple. Our main result in this paper is Theorem~\ref{main-th} establishing a one-to-one correspondence between rational $\GM$-actions $\varphi\colon \GM\times X\dashrightarrow X$ on $X$ and rational semisimple derivations $\partial\colon K_X\to K_X$ on $X$. The correspondence is as follows: we prove in Corollary~\ref{Cor5.2} that the image of $\varphi^*$ is contained in $K_X(t)\cap K_X[|t-1|]=\{r(t)\in K_X(t)\mid \operatorname{ord}_1(r)\geq 0\}$ and so we obtain $\partial$ from $\varphi$ via
$$\partial\colon K_X\to K_X \quad\mbox{given by}\quad  f\mapsto \operatorname{ev}_{1}\circ \frac{d}{dt}\circ \varphi^*(f)\,.$$
As for the other direction, letting $\sigma$ be the  isomorphism of formal power series rings given by the logarithmic power series
$$\sigma\colon K_X[|z|]\to K_X[|t-1|],\quad \mbox{given by} \quad z\mapsto \sum_{i\geq 1}(-1)^{i+1}\frac{(t-1)^i}{i}\,.$$
we recover the rational action $\varphi$ from $\partial$ via $\varphi^*=\sigma\circ\exp(z\partial)$.

As a consequence of our main result, we prove in Corollary~\ref{rational-slice} the existence of a rational slice for every rationally semisimple derivation, i.e., an element $s\in K_X$ such that $\partial(s)=s$. Moreover, we prove in Proposition~\ref{cor desc} that $K_X=K_X^{\GM}(s)$ and that under this isomorphism the derivation $\partial$ is given by $\partial=s\frac{d}{ds}$. Here, $K_X^{\GM}$ is the field of  invariant of the $\GM$-action.

Finally,  we provide  in Proposition~\ref{prop:Regular-actions} a characterization of regular actions of  the  multiplicative  group on  the  class  of  varieties that are proper over the spectrum of its ring of global regular functions. Such characterization agrees with the one recalled above in the particular case of affine varieties.

\subsection*{Acknowledgements}  

Part of this work was done during a stay of both authors at IMPAN in Warsaw. We would like to thank IMPAN and the organizers of the Simons semester ``Varieties: Arithmetic and Transformations'' for the hospitality.

\section{Rational $\GM$-action}

Let $\mu\colon\GM\times \GM\to \GM$ the morphism given by the group law  $(t_1,t_2)\mapsto t_1t_2$, and let  $\operatorname{e}_{\GM}\colon\,\operatorname{Spec}(k)\to \GM$ be the neutral element map. A rational action of the multiplicative group is a rational map $\varphi\colon\GM\times X \dashrightarrow X , (x,t)\mapsto t\cdot x$ such that the following diagrams are commutative:
\begin{eqnarray}
\xymatrix@=4.5em{ \GM\times \GM\times X \ar@{-->}[r]^{\mathrm{id}_{\GM}\times \varphi} \ar[d]_{\mu\times \mathrm{id}_X} & \GM\times X \ar@{-->}[d]^{\varphi} &  \mathrm{Spec}(k)\times X \ar[r]^{e_{\GM}\times \mathrm{id}_X} \ar[dr]_{\mathrm{pr}_2} & \GM\times X\ar@{-->}[d]^{\varphi} \\ \GM\times X \ar@{-->}[r]^{\varphi} & X & &  X.}
\label{eq:GeomDiagrams2}
\end{eqnarray}

We let $\dom(\varphi)$ be largest open subset of $\GM\times X$ where $\varphi$ is well defined. If  $(g,x)\in \dom(\varphi)$, we denote $\varphi(g,x)$ simply by $g\cdot x$.  The next lemma show that $\dom(\varphi)\cap (\{g\} \times X)$ is a non-empty open subset of  $ \{g\} \times \GM$. In particular for $g=1$, $\dom(\varphi)\cap (\{1\} \times X)$ is a open subset not empty. This will allow us to exhibit a criterion for the existence of rational action in terms of the function field of $X$.

\begin{lemma}\label{lem-demazure}
Let $X$ be an algebraic variety endowed with a rational $\GM$-action. For every fixed  $g\in \GM$, we define $V_{g}=\dom(\varphi) \cap(\{g\} \times X)$.
$V_{g}$ is a non-empty open subset of $\{g\} \times X \simeq X$ and the morphism
\begin{eqnarray*}
\varphi_{g}\colon V_{g} & \dashrightarrow & X\\
 x & \mapsto & \varphi(g,x)
\end{eqnarray*}
is dominant.
\end{lemma}

\begin{proof}
  Following Demazure in \cite{Dem70} we  consider the  morphism
 $u_{g}\colon \GM\to \GM$
, $h\mapsto gh^{-1}$ whose inverse is  $(u_{g})^{-1}\colon \GM\to \GM$
, $h\mapsto h^{-1}g$. We define $\beta_{g}=\varphi\circ(u_{g}\times \operatorname{id}_{X})\circ(\operatorname{pr}_1,\varphi)$,where $\operatorname{pr}_1$ is the projection in the first coordinate and $\operatorname{id}_{X}$ is the identity in $X$. As
\begin{center}
$\beta_{g}\colon\GM\times X \stackrel{ (\operatorname{pr}_{1},\varphi) }{\displaystyle\dashrightarrow} \GM\times X \stackrel{u_{g}\times\operatorname{Id}_{X}  }{\displaystyle\dashrightarrow} \GM\times X \stackrel{\varphi  }{\displaystyle\dashrightarrow} X$
\end{center}
We have $\beta_{g}\colon \GM\times X  \dashrightarrow X$, $(h,x)\mapsto  \varphi(g,x)$ is dominant for all choice of $g$.

Note that this map only depend of $x$.
Then $\dom(\beta_{g})=\GM \times U$, with $U \subseteq X$ an open subset. We have $\beta_{g}=\varphi_{g}\circ\operatorname{pr}_{2}$, where $\operatorname{pr}_2$ is the projection of $\GM\times X$ into the second coordinate. Hence,  $U=\operatorname{pr}_{2}(V_{g})$. Since $\varphi_{g}$ is a rational map and $V_{g}=\{g\}\times \dom(\varphi_{g})$, we obtain $V_{g}\not=\emptyset$.  Moreover $\beta_{g}=\varphi_{g}\circ\operatorname{pr}_{2}$, which means $\varphi_{g}$ is dominant because $\operatorname{Im}(\beta_{g})\subseteq \operatorname{Im}(\varphi_{g})$. 
\end{proof}

We can apply the case $g=1$ to obtain the following characterization.

\begin{corollary}\label{Cor5.2}
Let $\varphi\colon \GM\times X \dashrightarrow X$ be a rational action then $\operatorname{Im}(\varphi^{\ast})\subseteq \mathcal{O}_{\nu}$, where $\mathcal{O}_{\nu}=\{ r(t) \in K(t) \,\,| \, \,
\operatorname{ord}_{1}(r)\geq 0 \}$ is the discrete valuation ring of 
$K(t)$ and $\ord_1$ is the order of vanishing in $t=1$.
\end{corollary}

\begin{proof}
Let $V'\subseteq \GM\times X$ and $X'\subseteq X$ be affine open sets such that $\varphi|_{V'}\colon V'\to X'$ is regular. We let $\OO_{\GM\times X}(V')= A\subseteq K(t)$ and $\OO_X(X')= B\subseteq K$. By Lemma~\ref{lem-demazure}, we may and will assume $V'\cap V_1\neq \emptyset$ and so it is an dense open set of $\{1\}\times X$. Now, the composition $V'\cap V_1\to V'\to X'$ is dominant by Lemma~\ref{lem-demazure} and induces algebra homomorphisms
$$B\stackrel{\varphi^*}{\longrightarrow} A\longrightarrow A/A(t-1)\,,$$ 
and this composition is injective. This in turn shows that $\ord_1(\varphi^*(b))=0$ for every $b\in B$ different  from $0$. Since $K=\fract B$ we have $\ord_1(\varphi^*(f))=0$ for every $f\in K$, which proves the corollary.
\end{proof}

Given a rational action $\varphi\colon \GM\times X \dashrightarrow X$, for every fixed $g\in \GM$ we obtain a birational automorphism $$\varphi_{g}\colon X  \dashrightarrow X, \qquad x\mapsto  \varphi(g,x)$$
since $\varphi(1,x)=\operatorname{Id}_X$ and for every $g,g'\in \GM$ we have $\varphi_{g} \circ \varphi_{g'}(x)=\varphi_{gg'}(x)$. Moreover, the map $\GM\to \operatorname{Bir(X)}$ sending $g$ to $\varphi_g$ is a group homomorphism. Finally, a rational action  $\varphi\colon \GM\times X \dashrightarrow X$ such that $\dom(\varphi)=\GM\times X$ is a regular action.

\subsection{Criterion for existence of $\GM$-rational actions}

A $\GM$-rational action $\varphi\colon\GM\times X \dashrightarrow X$ in a variety $X$,is equivalent to the co-action homomorphism  $\varphi^{\ast}_t\colon K_X\to k(\GM\times X)=K_X(t)$ where the map $\varphi^{\ast}_t$  factors through the subalgebra $\mathcal{O}_{\nu}=\{ r(t) \in K_X(t) \,\,| \, \, \operatorname{ord}_{1}(r)\geq 0 \}$, with $\mathfrak{m}_{\nu}=\{ r(t) \in K_X(t) \,\,| \, \, \operatorname{ord}_{1}(r)> 0 \}$. The co-action morphism is characterized by the commutativity of the  follow diagrams:
 \begin{eqnarray} \label{comorphism}
\xymatrix@=4.5em{ K_X \ar@{->}[r]^{\varphi^{\ast}_{t_1}} \ar[d]_{\varphi^{\ast}_{t_2}} & K_X (t_1) \ar@{->}[d]^{t_1\mapsto t_1t_2} & & K_X\simeq \mathcal{O}_{\nu}/ \mathfrak{m}_{\nu} \ar[r]^{\varphi^{\ast}_t} \ar[dr]_{\operatorname{Id}_{K_X}} & \mathcal{O}_{\nu}\ar@{->}[d]^{\operatorname{ev}_1} \\ K_X(t_2) \ar@{->}[r]^{\widetilde{\varphi}} &  K_X(t_1,t_2) & & & K_X\simeq \mathcal{O}_{\nu}/ \mathfrak{m}_{\nu}}
\label{eq:GeomDiagrams3}
\end{eqnarray}

The following proposition is classical.

\begin{proposition}\label{prop}
Le $X$ be a variety. $X$ admits a nontrivial rational $\GM$-action if only if it is birationally isomorphic to $Y\times\mathbb{P}^1$ for some $k$-variety $Y$.
\end{proposition}

\begin{proof}
See \cite{Ros56,Mat63,Pop16}.
\end{proof}

Mimicking the case of $\GA$-actions on affine geometry, we we have the following definition:

\begin{definition} \label{def:rational-slice}
 Let $X$ be a $\GM$-variety, an element $s\in K_X$ such that $\varphi^{\ast}_t(s)=ts$ is called rational slice.
\end{definition}

For each faithful rational action, a rational slice always exists. Indeed, since the action is faithful, there exists  two semi-invariant $a,b\in K_X$ whose weights $n,m$ are relatively prime, i.e.,  $\varphi^{\ast}_t(a)=t^{n}a$ and $\varphi^{\ast}_t(b)=t^{m}b$. By Bezout theorem we have that there exist $c,d\in \mathbb{Z}$ such that $nd+mc=1$. Hence $s=a^db^c$ satisfies $\varphi^{\ast}_t(s)=st$  and so is a rational slice.

A derivation on $k$-algebra $A$, is a linear map $\partial\colon A\to A$ such that satisfy the Leibniz rules, $\partial(ab)=a\partial(b)+b\partial(a)$. We define the kernel of a derivation $\partial$ as its kernel as a linear map, i.e.,  $\operatorname{ker}(\partial):=\{a\in A| \partial(a)=0\}$. The set of derivations over $A$ is denoted by $\operatorname{Der}(A)$.
 We say $\partial$ is a semisimple derivation on $A$ if there exists a basis $\{a_i\}_{i\in I}$ of $A$ as $k$-vector space of eigenvalues such that $\partial(a_i)\in ka_i$, we will focus in the subset of semisimple derivations whose eigenvalues are integers numbers. For more details over semisimple derivations see \cite{Now01,Now94}. Is known that the regular $\GM$-actions are in correspondence with the set of semisimple derivations whose eigenvalues are integers numbers.

Analogously with the definition of semisimple derivations over a $k$-algebra, we will give a definition of derivation over $K_X$, we will call rational semisimple.


\begin{definition} \label{def:rat-semisimple}
 Let be $K_X$ the field of rational function associated to algebraic variety $X$, $\partial\in \operatorname{Der}(K_X)$. We say $\partial$ is rational semisimple if there exist a generating set $\{a_i\}_{i\in I}$ of $K_X$ (as field) such that $\partial(a_i)=\lambda_i a_i$ with $\lambda_i\in \mathbb{Z}$.
\end{definition}

Let $\partial\colon K_X\to K_X$ a $k$-derivation. Denoting  the $i$-th iteration of $\partial$ by $\partial^{i}$ and $\partial^0$ is the identity map, we define the exponential map 
$$\exp(z\partial)\colon K_X\to K_X[|z|],\quad  \mbox{given by} \quad f\mapsto \sum_{i\geq 0}\frac{z^i\partial^{i}(f)}{i!}\,.$$
Furthermore, since $\partial$ is a $k$-derivation, the following proposition shows that $\exp(z\partial)$ is a $k$-rings homomorphism.

\begin{proposition}\label{prop2.7}
Let be $\partial$ a $k$-derivation on $K_X$, then exponential map $\exp(z\partial)\colon K_X\to K_X[|z|]$ is a $k$-ring homomorphism.
\end{proposition}

\begin{proof}
The exponential map $\exp(z\partial)$ is $k$-linear since $\partial$ is $k$-linear. As for the product structure, we have 
\begin{align*}
\exp(z\partial)(fg)&=\sum_{i\geq 0}
   \frac{\partial^i(fg)z^{i}}{i!}\\
&=\sum_{i \geq 0}
   \frac{1}{i!}\displaystyle\left(\sum^{i}_{j = 0}{n \choose j}\partial^{j}(f)\partial^{i-j}(g)\right)z^{i}\\
&=\sum_{i \geq 0}
   \displaystyle\left(\sum^{i}_{j = 0}\frac{\partial^{j}(f)}{j!}\frac{\partial^{i-j}(g)}{(i-j)!}\right)z^{i}\\
&=\sum_{i\geq 0}
   \frac{\partial^i(f)z^{i}}{i!}\sum_{l\geq 0}
   \frac{\partial^l(g)z^{l}}{l!}\\
 \exp(z\partial)(fg)&=\exp(z\partial)(f)\exp(z\partial)(g)
\end{align*}
\end{proof}

Furthermore, there is an isomorphism of formal power series rings
\begin{align*}\sigma\colon K_X[|z|]\to K_X[|t-1|],\quad \mbox{given by} \quad z\mapsto \sum_{i\geq 1}(-1)^{i+1}\frac{(t-1)^i}{i}\,.
\end{align*}
The inverse of this isomorphism is given by
\begin{align*}
\sigma^{-1}\colon K_X[|t-1|]\to K_X[|z|],\quad \mbox{given by} \quad t-1\mapsto \sum_{i\geq 1}\frac{z^i}{i!}\,.
\end{align*}

This corresponds to the logarithmic series and its inverse is the exponential series. In the sequel we will show that the image of $\sigma\circ\exp(z\partial)$ is contained in $K_X(t-1)=K_X(t)$ and that $\sigma\circ\exp(z\partial)$ is the comorphism of a rational $\GM$-action if and only if $\partial$ is a rational semisimple derivation. With this in view, for every we denote $\phi^*_{\partial}=\sigma\circ\exp(z \partial)$ so that its comorphism corresponds to a map $\phi_\partial\colon \GM\times X\dashrightarrow X$.

\begin{lemma} \label{contained-action}
Let be  $\partial\colon K_X\to K_X$ a rational derivation. Assume that  the image of $\phi^*_\partial$ is contained in $K_X(t)$, then $\phi_{\partial}$ is an action of the multiplicative group.
\end{lemma}
\begin{proof}

From Proposition~\ref{prop2.7} we have $\phi^*_{\partial}$ is a ring homomorphism. Let 
$$t_1=\sum_{i\geq 0}\frac{z_1^i}{i!}\quad\mbox{and}\quad t_2=\sum_{i\geq 0}\frac{z_2^i}{i!}$$
We have the identity
$$t_1t_2=\displaystyle\sum_{i\geq 0}\frac{1}{i!}(z_1+z_2)^i$$

By Lemma~\ref{lem-demazure}, for every $t\in \GM$ we can specialize $\phi_\partial^*$ to obtain a field automorphism $\phi_\partial^*(t)\colon K_X\to K_X$. We extend $\partial$ to a derivation $\partial\colon K_X(z_1,z_2)\to K_X(z_1,z_2)$ by setting $\partial(z_1)=\partial(z_2)=0$. Since now the derivations $z_1\partial$ and $z_2\partial$ commute, by \cite[Proposition~2.4.2]{Now94} we have 
\begin{align*}
\phi^*_{\partial}(t_1t_2)(f)&=\sigma\circ\exp((z_1+z_2)\partial)(f)\\
&=\sigma\circ\big(\exp(z_1\partial)\circ \exp(z_2\partial)\big)(f)\\
&=\phi^*_{\partial}(t_1)\circ \phi^*_{\partial}(t_2)(f)
\end{align*}

Furthermore, by definition of $\phi^*_\partial\colon K_X\to K_X[|t-1|]$ the composition 
\begin{align*}
\operatorname{ev}_1\circ\phi^*_\partial&=\operatorname{ev}_1\circ\sigma\circ\exp(z\partial)\\
&=\operatorname{ev}_0\circ\exp(z\partial)\\
&=\operatorname{ev}_0\circ\sum_{i\geq 0}\frac{z^i\partial^i}{i!}\\
\operatorname{ev}_1\circ\phi^*_\partial &=\operatorname{Id}_{K_X}.
\end{align*}  
Hence, by \eqref{comorphism} we have that $\phi_\partial^*$ is the comorphism of a multiplicative group action.
\end{proof}

\begin{proposition}\label{prop2.8} 
The following are equivalent.
\begin{enumerate}[$(i)$]
    \item $\partial\colon K_X\to K_X$ is rational semisimple.
    \item The image of $\phi^*_{\partial}$ is contained in $K_X(t)$. 
\end{enumerate}
\end{proposition}

\begin{proof}
Assume first that $\partial$ is rational semisimple. Hence, there exists a set of field generators $\{a_i\}_{i\in I}$ of $K_X$ such that
$\partial(a_j)=\lambda_ja_j$ with $\lambda_j\in \Z$. For these generators we have $$\phi^*_{\partial}(a_j)=\sigma\circ\exp(z\partial)(a_j)=\sigma\left(\sum_{i\geq 0}\frac{z^i\partial^i(a_j)}{i!}\right)=\sigma\left(\sum_{i\geq 0}\frac{z^i\lambda^i_ja_j}{i!}\right)=\sigma\left(\sum_{i\geq 0}\frac{(z\lambda_j)^i}{i!}a_j\right)=t^{\lambda_j} a_j\,.$$ 

Let now $f,g\in k[a_i,i\in I]$ with $g\neq 0$. Now, 
\begin{align*}
\phi^*_{\partial}\left(\frac{f}{g}\right)    &=
\sigma\circ\exp(z\partial)\left(\frac{f}{g}\right)\\
&=\sigma\circ\exp(z\partial)\left(\frac{f(a_i,i\in I)}{g(a_i,i\in I)}\right) \\
    &=\frac{\sigma\circ\exp(z\partial)(f(a_i,i\in I))}{\sigma\circ\exp(z\partial)(g(a_i,i\in I))} \\
        &=\frac{f(t^{\lambda_i} a_i,i\in I)}{g(t^{\lambda_i}a_i,i\in I)}\in K_X(t)
\end{align*}
This proves $(i)\to (ii)$. Inspired by Koshevoii \cite{Kos67}, to prove the converse assertion, we let $f\in K_X$ and we let
$$\phi^*_{\partial}(f)=t^\ell\cdot \frac{\sum_i a_it^i}{\sum_i b_it^i},$$
with $\ell \in \ZZ$ and $a_0,b_0\neq 0$. We also may and will assume that the representation is irreducible meaning that $\sum_i a_it^i$ and $\sum_i b_it^i$ are relatively prime. Furthermore, such a   representations of $\phi^*_{\partial}(f)$ is unique if we further assume that $b_0=1$.

By Lemma~\ref{contained-action}, we have $\phi^*_{\partial}(s)\circ\phi^*_{\partial}(t)(f)=\phi^*_{\partial}(st)(f)$. This yields
$$t^\ell\dfrac{\sum_{i} \phi^*_{\partial}(s)(a_i)t^i}{\sum_{i} \phi^*_{\partial}(s)(b_i)t^i}=t^\ell \dfrac{\sum_{i} a_i(st)^i}{\sum_{i} b_i(st)^i}=t^\ell \dfrac{\sum_{i} (a_is^i)t^i}{\sum_{i} (b_is^{i})t^i}.$$
This yields
$\phi^*_{\partial}(a_j)=a_jt^j$ and  $\phi^*_{\partial}(b_j)=b_jt^j$. In particular, this implies that 
$$\phi^*_{\partial}(a_j)=\sigma\circ\exp(z\partial)(a_j)=\sigma\left(\sum_{i\geq 0}\frac{z^i\partial^i(a_j)}{i!}\right)\,.$$ 
while 
$$t^{j}a_j=\sigma\left(\sum_{i\geq 0}\frac{(zj)^i}{i!}\right)\cdot a_j=\sigma\left(\sum_{i\geq 0}\frac{(zj)^i}{i!}a_j\right)=\sigma\left(\sum_{i\geq 0}\frac{z^ij^ia_j}{i!}\right)$$
We conclude that
$$\sum_{i\geq 0}\frac{z^i\partial^i(a_j)}{i!}=\sum_{i\geq 0}\frac{z^ij^ia_j}{i!}$$
In particular, taking equality in the term with $i=1$ we obtain that $\partial(a_j)=ja_j$. Finally, since $f=\dfrac{\sum_{i} a_j}{\sum_{i} b_i}$ we obtain that the set 
$$\left\{a_i,b_i\in K_X\mid t^\ell\cdot \phi_\partial^*(f)=\frac{\sum_i a_it^i}{\sum_i b_it^i} \mbox{ for some } f\in K_X\right\}$$
generates $K_X$ and so $\partial$ is a  rational semisimple derivation.
\end{proof}

Given a rational $\GM$-action $\varphi$ on $X$ with comorphism $\varphi^*\colon K_X\to K_X(t)$, we define the map 
$$D_\varphi\colon K_X\to K_X \quad\mbox{given by}\quad  f\mapsto \operatorname{ev}_{1}\circ \frac{d}{dt}\circ \varphi^*(f)\,,$$
generalizing the usual definition of the infinitesimal generator of a 1-parameter group action. The following lemmas will be required in our proof of our main result.

\begin{lemma} \label{derivation-from-action}
Let $\varphi$ be a rational $\GM$-action on $X$ with comorphism $\varphi^*\colon K_X\to K_X(t)$. Then the follwing hold:
\begin{enumerate} [$(i)$]
    \item The map $D_\varphi$ is a derivation.
    \item We have an equality $D_\varphi=\displaystyle \operatorname{ev}_1\circ\frac{d}{dt}\circ\varphi^*=\operatorname{ev}_0\circ\frac{d}{dz}\circ\sigma^{-1}\circ\varphi^*$.
\end{enumerate}

\end{lemma}

\begin{proof}
By Corollary \ref{Cor5.2}, we know that the image of $\varphi^*$ is contained in $K_X[|t-1|]$. Given $f$ and $g$ in $K_X$, we let $\varphi^*(f)=\sum_{i\geq 0}a_i(t-1)^i$ and $\varphi^*(g)=\sum_{i\geq 0}b_i(t-1)^i$. We have $a_0=f$ and $b_0=g$ since $\varphi^*$ is the comorphism of a $\GM$-action. Moreover, we have 
\begin{align} \label{eq:D-is-a1}
    D_\varphi(f)=\operatorname{ev}_1\circ\frac{d}{dt}\circ\varphi^*(f)=\operatorname{ev}_1\left(\sum_{i\geq 1}a_i\cdot i(t-1)^{i-1}\right)=a_1\,,
\end{align}
so that the map $D_\varphi(f)$ corresponds to the first order term of $\varphi^*(f)$. To prove $(i)$, remark that the composition $D_\varphi=\operatorname{ev}_1\circ \frac{d}{dt}\circ\varphi^*$ is $k$-linear and maps $k$ to $0$. Furthermore, the term of first order term of $fg$ is $a_0b_1+a_1b_0$. This yields the Leibniz rule since $$D_\varphi(fg)=a_0b_1+a_1b_0=fD_\varphi(g)+D_\varphi(f)g\,.$$
Assertion $(ii)$ follows by the following straightforward computation:
\begin{align*}
 \operatorname{ev}_0\circ\frac{d}{dz}\circ\sigma^{-1}\circ\varphi^*(f)&=\operatorname{ev}_0\circ \frac{d}{dz}\circ \sigma^{-1}\left(\sum_{i\geq 0}a_i\cdot (t-1)^i\right)\\
 &=\operatorname{ev}_0\circ \frac{d}{dz}\left(\sum_{i\geq 0}a_i \left(\sum_{j\geq 1}\frac{z^j}{j!}\right)^i\right)\\
 &=\operatorname{ev}_0\left[\sum_{i\geq 1}a_i\cdot i \left(\sum_{j\geq 1}\frac{z^j}{j!}\right)^{i-1}\cdot\left(\sum_{j\geq 0}\frac{z^j}{j!}\right)\right]=a_1=D_\varphi(f)\,.
\end{align*}
\end{proof}

We will now prove that the map $\psi^*:=\sigma^{-1}\circ\varphi^*\colon K_X\to K_X[|z|]$ is the germ of a rational $\GA$-action in $X$. Indeed, letting $s$ and $w$, in the following diagram put as subscript the transendental element over $K_X$ in the target ring.
\begin{align*}
    & K_X\stackrel{\varphi_t^*}{\longrightarrow} K_X[|t-1|]\stackrel{\sigma^{-1}_z}{\longrightarrow} K_X[|z|] \\
    & K_X\stackrel{\varphi_s^*}{\longrightarrow} K_X[|s-1|]\stackrel{\sigma^{-1}_w}{\longrightarrow} K_X[|w|] 
\end{align*}
We also let $\psi_z^*:=\sigma_z^{-1}\circ\varphi_t^*$ and 
$\psi_w^*:=\sigma_w^{-1}\circ\varphi_s^*$. With these definitions we now proof the following lemma.
\begin{lemma} \label{germ-Ga}
With the above notation, we have $\operatorname{ev}_0\circ\psi^*_z=\operatorname{Id}_{K_X}$ and  $\psi_z^*\circ \psi_w^*=\psi_{z+w}^*$.
\end{lemma}
\begin{proof}
Letting $f\in K_X$ we assume 
$$\varphi^*(f)=\sum_{i\geq 0}a_i(t-1)^i \quad \mbox{so that} \quad \psi^*_z(f)=\sum_{i\geq 0}a_i\left(\sum_{j\geq 1}\frac{z^j}{j!}\right)^i\,.$$ Since $\varphi^*$ is the comorphism of a $\GM$-action we have $a_0=f$ and so $\operatorname{ev}_0\circ\,\psi^*_z(f)=a_0=f$. To prove the second assertion, remark that
\begin{align*}
    \psi^*_z\circ\psi^*_w(f)&=\sigma^{-1}_z\circ\varphi^*_t\circ \sigma^{-1}_w\circ\varphi^*_s\\
    &=\sigma^{-1}_z\circ\sigma^{-1}_w\circ\varphi^*_t\circ \varphi^*_s(f)\\
    &=\sigma^{-1}_z\circ\sigma^{-1}_w\circ \varphi^*_{st}(f)\\
    &=\sigma^{-1}_z\circ\sigma^{-1}_w\left(\sum_{i\geq 0}a_i(st-1)^i\right)\\
    &=\sigma^{-1}_z\circ\sigma^{-1}_w\left(\sum_{i\geq 0}a_i\bigg[\big(s-1+1\big)\big(t-1+1\big)-1\bigg]^i\right) 
\end{align*}
Applying now $\sigma^{-1}_z\circ\sigma^{-1}_w$ amounts to replace $t-1$ by $\sum_{i\geq 1}\frac{z^i}{i!}$ and $s-1$ by $\sum_{i\geq 1}\frac{w^i}{i!}$. Hence, we obtain
    \begin{align*}
    \psi^*_z\circ\psi^*_w(f)&=\sum_{i\geq 0}a_i\left[\left(\sum_{j\geq 1}\frac{w^j}{j!}+1\right)\left(\sum_{j\geq 1}\frac{z^j}{j!}+1\right)-1\right]^i\\
    &=\sum_{i\geq 0}a_i\left[\left(\sum_{j\geq 0}\frac{w^j}{j!}\right)\left(\sum_{j\geq 0}\frac{z^j}{j!}\right)-1\right]^i\\    
\end{align*}
Now, by the usual properties of the exponential sum, we obtain
\begin{align*}
    \psi^*_z\circ\psi^*_w(f)&=\sum_{i\geq 0}a_i\left[\left(\sum_{j\geq 0}\frac{(w+z)^j}{j!}\right)-1\right]^i\\
    &=\sum_{i\geq 0}a_i\left[\left(\sum_{j\geq 1}\frac{(w+z)^j}{j!}\right)\right]^i\\
   \psi^*_z\circ\psi^*_w(f)     &=\psi^*_{w+z}(f)
\end{align*}
\end{proof}

\begin{lemma} \label{derivaciones-z-t}
Let $\varphi$ be a rational $\GM$-action on $X$ with comorphism $\varphi^*\colon K_X\to K_X(t)$. Then the iterations $D^i_\varphi$ satisfy $\displaystyle D^i_\varphi=\operatorname{ev}_{0}\circ \frac{d^i}{dz^i}\circ \sigma^{-1}\circ\varphi^*$
\end{lemma}

\begin{proof}
We can now use the argument as in \cite[Proposition~4.10]{Dai03} applied to $\psi^*=\sigma^{-1}\circ\varphi^*$. For the convenience of the reader we copy the argument here. Letting $\psi^*_z(f)=\sum_{i\geq 0}a_iz^i$, we have
\begin{align*}
    \sum_{i\geq 0}\psi^*_{z}(a_i)w^i=\psi^*_z\circ\psi^*_w(f)=\psi^*_{w+z}(f)
    &=\sum_{\ell\geq 0}a_\ell(z+w)^\ell \\
    &=\sum_{\ell\geq 0}a_\ell\sum_{i+j=\ell\geq 0}\binom{\ell}{i}z^jw^i\\
    &=\sum_{i\geq 0}\left( \sum_{j\geq 0}a_{i+j}\binom{i+j}{i}z^j\right)w^i
\end{align*}
Hence we obtain 
$$\psi^*(a_i)=\sum_{j\geq 0}a_{i+j}\binom{i+j}{i}z^j\mbox{ and in particular } \psi^*(a_1)=\sum_{j\geq 0}a_{j+1}(j+1)z^j=\sum_{j\geq 1}a_{j}jz^{j-1}$$

Since $D_{\varphi}(f)=a_1$ by \eqref{eq:D-is-a1}, we have $\varphi^*\circ D_\varphi=\frac{d}{dz}\circ\varphi^*$. Indeed,
$$\psi^*(D_{\varphi}(f))=\psi^*(a_1)=\sum_{j\geq 1}a_{j}jz^{j-1}=\frac{d}{dz}\left(\sum_{j\geq 0}a_jz^j\right)=\frac{d}{dz}(\psi^*(f))$$
Hence, we have $\psi^*\circ D^i_{\varphi}=\frac{d^i}{dz^i}\circ\psi^*$ for all $i\geq 0$. Composing on the left with $\operatorname{ev}_0$ we obtain
$$D^i_\varphi=\operatorname{ev}_0\circ \frac{d^i}{dz^i}\circ\sigma^{-1}\circ\varphi^*\,,$$
since $\operatorname{ev}_0\circ\,\psi^*=\operatorname{Id}_{K_X}$ by Lemma~\ref{germ-Ga}.
\end{proof}

In the following proposition, we show that the derivation $D_\varphi$ is rational semisimple.

\begin{proposition}\label{prop-comp-1}
Let $\varphi\colon \GM\times X\dasharrow X$ be a $\GM$-rational action on $X$.  Then the following hold
\begin{itemize}
    \item[$(i)$] $D_\varphi$ is a rational semisimple derivation.
    \item[$(ii)$] The composition $\phi^*_{D_\varphi}=\sigma\circ\exp(zD_\varphi)$ equals $\varphi^*$.
\end{itemize}
\end{proposition}

\begin{proof}
The assertion $(i)$ follows directly from $(ii)$ and Proposition~\ref{prop2.8} since in this case the image of $\phi^*_{D_\varphi}$ equals the image of $\varphi^*$ which is contained in $K_X(t)$.

To prove $(ii)$, let $f\in K_X$. By Corollary~\ref{Cor5.2}, we have that $\varphi^*(f)\in K_X[|t-1|]$ and so $\sigma^{-1}\circ \varphi^*(f)\in K_X[|z|]$. Let $\sigma^{-1}\circ\varphi^*(f)=\sum_{j\geq 0}a_jz^j$. By Lemma~\ref{derivaciones-z-t} we have that 
\begin{align*}
    \phi^*_{D_\varphi}(f)&=\sigma\circ\sum_{i\geq 0}\frac{z^iD^i_{\varphi}(f)}{i!} \\
     &=\sigma\circ\sum_{i\geq 0}\frac{z^i}{i!}\cdot \operatorname{ev}_0\circ\frac{d^i}{dz^i}\circ\sigma^{-1}\circ\varphi^*(f) \\
     &=\sigma\circ\sum_{i\geq 0}\frac{z^i}{i!}\cdot \operatorname{ev}_0\circ\frac{d^i}{dz^i}\left(\sum_{j\geq 0}a_jz^j\right) \\
     &=\sigma\circ\sum_{i\geq 0}\frac{z^i}{i!}\cdot \operatorname{ev}_0\circ\left(\sum_{j\geq i}a_j\frac{j!}{(j-i)!}z^{j-i}\right) \\
     &=\sigma\circ\sum_{i\geq 0}\frac{z^i}{i!}\cdot a_i \cdot i! \\
     &=\sigma\circ\sum_{i\geq 0}a_i z^i=\sigma\circ\sigma^{-1}\circ\varphi^*(f)=\varphi^*(f)
\end{align*}
\end{proof}

\begin{proposition}\label{prop-comp-2}
Let $\partial\colon K_X\to K_X$ be a rational semisimple derivation. Then the composition $D_{\phi_\partial}=\operatorname{ev}_1\circ\frac{d}{dt}\circ\sigma\circ\exp(z\partial)$ equals $ \partial$.
\end{proposition}

\begin{proof}
Let $\partial$ be a rational semisimple derivation. Then by Lemma~\ref{derivation-from-action}~(ii), we have
\begin{align*}
D_{\phi_\partial}=\operatorname{ev}_1\circ\frac{d}{dt}\circ\sigma\circ\exp(z\partial)=\operatorname{ev}_0\circ\frac{d}{dz}\circ\sigma^{-1}\circ\sigma\circ\exp(z\partial)=\operatorname{ev}_0\circ\frac{d}{dz}\circ\exp(z\partial)
\end{align*}
Letting now $f\in K_X$ we obtain
\begin{align*}
D_{\phi_\partial}(f)&=\operatorname{ev}_0\circ\frac{d}{dz}\circ\exp(z\partial)(f) 
    =\operatorname{ev}_0\circ\frac{d}{dz}\circ \sum_{i\geq 0}\frac{z^i\partial^i(f)}{i!}
    =\operatorname{ev}_0\circ \sum_{i\geq 1}\frac{iz^{i-1}\partial^i(f)}{i!}=\partial(f)\,.
\end{align*}
This proves the proposition.
\end{proof}

The following is our main theorem in this paper establishing a one-to-one correspondence between rational $\GM$-actions on $X$ and rational semisimple derivations on $K_X$.

\begin{theorem} \label{main-th}
Let $X$ be an algebraic variety. There exists a one-to-one correspondence between the rational $\GM$-actions over $X$ and rational semisimple derivations on $K_X$ given by
\begin{align*}
\Big\{\mbox{Rational semisimple derivations on $K_X$}\Big\} \quad &\longleftrightarrow \quad\Big\{\mbox{Rational $\GM$-actions on $X$}\Big\} \\
\partial \quad & \longrightarrow \quad \phi_\partial \\
D_\varphi \quad &\longleftarrow \quad \varphi
\end{align*}
\end{theorem} 

\begin{proof}
Let $X$ be an algebraic variety. If $\partial\colon K_X\to K_X$ is a rational semisimple derivation, then $\phi^*_\partial$ is a rational $\GM$-action by Lemma~\ref{contained-action} and Proposition~\ref{prop2.8}. On the other hand, if $\varphi$ is a rational $\GM$-action on $X$, then, by Proposition~\ref{prop-comp-1}~$(i)$, we have that $D_\varphi$ is rational semisimple. The fact that these maps are mutually inverse to each other is proven in Proposition~\ref{prop-comp-1}~(ii) and Proposition~\ref{prop-comp-2}.
\end{proof}

\section{Examples and applications}

In this section we provide several examples and applications of our main theorem. To perform the computations in the sequel, we need the following technical lemma proving that conjugation of a rational semisimple derivation by an automorphism $\varphi$ amounts to conjugation of the corresponding $\GM$-action by the same automorphism $\varphi$.

\begin{lemma}\label{lemma-ejemplos}
Letting $\partial\colon K_X\to K_X$ be a rational semisimple derivation and $\varphi^*\colon K_X\to K_X$ be a $k$-automorphim, we have  
$$\phi^*_{\varphi^*\circ \partial\circ(\varphi^*)^{-1}}=\varphi^*\circ\phi^*_{\partial}\circ(\varphi^*)^{-1}\,$$

\end{lemma}

\begin{proof}
Since $(\varphi^*\circ \partial\circ (\varphi^*)^{-1})^i=\varphi^*\circ \partial^i\circ(\varphi^*)^{-1}$ and $\sigma$ commutes with $\varphi^*$, we have 

\begin{align*}
\phi_{\varphi^*\circ \partial\circ (\varphi^*)^{-1}}&=\sigma\circ\sum_{i\geq 0}\frac{z^i(\varphi^*\circ \partial\circ (\varphi^*)^{-1})^i}{i!}
=\sigma\circ \varphi^*\circ\left(\sum_{i\geq 0}\frac{z^i\partial^i}{i!}\right)\circ(\varphi^*)^{-1}
=\varphi^*\circ \phi^*_{\partial}\circ (\varphi^*)^{-1}  
\end{align*}
\end{proof}

\begin{example}
Letting $X=\AF^2=\spec k[x,y]$ we let $E$ be the Euler derivation given by  
$$E=ax\frac{\partial}{\partial x}+by\frac{\partial}{\partial y} \quad \mbox{with}\quad  a,b\in \mathbb{Z}\,.$$
This derivation is a regular semisimple derivation corresponding to the linear $\GM$-action on $X$ given by
$$\GM\times X\to X \quad\mbox{where}\quad (t,(x,y))\mapsto (t^ax,t^by)\,.$$
If we conjugate $E$ with the birational map  
$$\varphi\colon X\dashrightarrow X \quad\mbox{given by}\quad (x,y)\mapsto  ((x-1)(y-1)+1,y)\,,$$
whose inverse is 
$$\varphi^{-1}\colon X\dashrightarrow X \quad\mbox{given by}\quad (x,y)\mapsto  \left(\frac{x-1}{y-1}+1,y\right)\,.$$
We define the rational semisimple derivation $\partial=\varphi^*\circ E\circ(\varphi^*)^{-1}$. A straightforward computation shows that
$$\partial=\left( \frac{a((x-1)(y-1)+1)-b(x-1)y}{y-1}\right)\frac{\partial}{\partial x} + by \frac{\partial}{\partial y}$$

By Lemma~\ref{lemma-ejemplos}, we have $\phi^*_{\partial}=\varphi^*\circ \phi^*_{E}\circ (\varphi^*)^{-1}$. More explicitly we obtain:
$$\phi_{\partial}\colon \GM\times X\dashrightarrow X \quad\mbox{given by}\quad   (x,y)\mapsto\left(\frac{t^a((x-1)(y-1)+1)-1}{t^by-1}+1,t^by\right)\,.$$

We can recover the rational derivation $\partial$ by computing $\operatorname{ev}_1\circ\frac{d}{dt}\circ\phi^*_{\partial}$. Indeed, a tedious computation shows that 
\begin{align*}
\frac{d}{dt}(\phi^*_{\partial}(x))&=\frac{\left(at^{a-1}[(x-1)(y-1)+1]\right)(t^by-1)-bt^{b-1}y[t^a[(x-1)(y-1)+1]-1]}{(t^by-1)^2} \\
\frac{d}{dt}(\phi^*_{\partial}(y))&=\frac{d}{dt}(t^by)=bt^{b-1}y\,.
\end{align*}
So that
 \begin{align*}
 \operatorname{ev}_1\circ \frac{d}{dt}\circ\phi^*_{\partial}(x)&=\frac{a((x-1)(y-1)+1)-b(x-1)y}{y-1} \\
 \operatorname{ev}_1\circ\frac{d}{dt}\circ\phi^*_{\partial}(y)&=by\,,
\end{align*}
recovering the initial derivation $\partial$.
\end{example}

Recall that a rational slice of a rational $\GM$-action $\varphi$ is a function $s\in K_X$ such that $\varphi^*(s)=ts$. We can also characterize slices in terms of the corresponding rational semisimple derivation as we show in the following lemma.

\begin{corollary} \label{rational-slice}
 Letting $X$ be an algebraic variety with fields of the rational function $K_X$, we let $\partial\colon K_X\to K_X$ be a rational semisimple derivation. Then $s\in K_X$ is a rational slice of $\phi^*_\partial(t)$ if and only if  $\partial(s)=s$.
\end{corollary} 

\begin{proof}
Assume first that $\partial(s)=s$. Then $\partial^i(s)=s$ for every $i\geq 0$ and so we have
\begin{align*}
    \phi^*_{\partial}(s)&=\sigma\circ\exp{z\partial}(s)=s\cdot\sigma\left(\sum_{i\geq 0}\frac{z^i}{i!}\right)=ts\,.
\end{align*}
Now if $\phi^*_{\partial}(s)=ts$ then $\partial(s)=\operatorname{ev}_1\circ\frac{d}{dt}\circ\phi^*_{\partial}(s)=s$
\end{proof}

Furthermore, a slice provide a ruling of the field $K_X$ over the field of invariants $K^{\GM}_X$ of the $\GM$-action.

\begin{proposition}\label{cor desc}
If $s$ is a rational slice for an faithful rational $\GM$-action $\varphi\colon \GM\times X\dashrightarrow X$, then $s$ is transcental over $K^{\GM}_X$,   $K_X=K^{\GM}_X(s)=(\ker\partial)(s)$ and $\partial=s\frac{d}{d s}$ on $K^{\GM}_X(s)$.
\end{proposition}

\begin{proof}
We define the set $T=\{a\in K^*_X \mid \varphi^*(a)=t^{i}a \mbox{ with } i \in \mathbb{Z}\}$. In the proof Proposition~\ref{prop2.8} we proved that $T$ generates $K_X$. Moreover, $T$ is a group under multiplication. We have $\varphi^*(a)=t^0a$ for all $a\in K^{\GM}_X$. Therefore $T_0:=K^{\GM}_X\setminus\{0\}$ is a subgroup of $T$. Let now $T/T_0$, given $a,b\in T$ such that $\varphi^*_t(a)=t^ia$ and $\varphi^*_t (b)=t^ia$, for some $i\in\ZZ$, we have $\varphi^*(ab^{-1})=ab^{-1}$ implies $ab^{-1}\in T_0$. Hence,  $a$ and $b$  differ by a element of $T_0$. 

We define the group homomorphism   $T\to \ZZ$ given by $a\mapsto i$ where $\varphi^*(a)=t^ia$. This homomorphism is surjective by the existence of a rational slice, see Definition~\ref{def:rational-slice} and below. Moreover, its kernel is $T_0$. We conclude that $T/T_0\simeq\ZZ$ with a rational slice $s$ as generator. Since $T_0$ generated the field $K^{\GM}_X$ and $T$ generates the field $K_X$ we obtain that $K_X=K^{\GM}_X(s)$.

Assume now that $s$ is algebraic over $K^{\GM}_X$, i.e., assume that there exists a non trivial polynomial $P\in K^{\GM}_X[x]$ such that $P(s)=0$, then $\varphi^*(P(s))=P(ts)=0$. This is a contradiction since $t$ is is transcendental over $K_X$ and so the same holds over the subfield $K^{\GM}_X$. We conclude that therefore $s$ is transcendental over $K^{\GM}_X$. Finally, the structure of $\partial$ given as $\partial=s\frac{d}{d s}$ on $K^{\GM}_X(s)$ follows since $\partial(K^{\GM}_X)=0$ and $\partial(s)=s$.
\end{proof}

As in the rational case treated above, given a regular $\GM$-action $\varphi$ on $X$, we define a derivation of the structure sheaf $D_\varphi\colon \OO_X\to \OO_X$ given over every affine open set $U\subseteq X$ by
$$D_\varphi\colon \OO_X(U)\to \OO_X(U) \quad\mbox{given by}\quad  f\mapsto \operatorname{ev}_{1}\circ \frac{d}{dt}\circ \varphi^*(f)\,.$$
Any derivation $\partial\colon \OO_X\to\OO_X$ induces a derivation $K_X\to K_X$ simply by extending $\partial\colon \OO_X(U)\to\OO_X(U)$ to the field of fractions $\fract(\OO_X(U))=K_X$ for any affine open set via the Leibniz rule. We denote this derivation also by the same symbol $\partial\colon K_X\to K_X$.

In the next proposition, we characterize the derivations of the structure sheaf $X$ that come from a regular $\GM$-action in the case where $X$ is semi-affine. Recall that a variety $X$ is called semi-affine if the canonical
morphism $X\to \spec \OO_{X}(X)$
is proper. In this case $\mathcal{O}_{X}(X)$ is finitely generated and so $\spec\mathcal{O}_{X}(X)$ is an affine variety \cite[corollary 3.6]{GoLa73}. For
instance, complete or affine $k$-varieties are semi-affine, blow-ups of semi-affine varieties are also semi-affine.

\begin{proposition}\label{prop:Regular-actions}
Regular $\GM$-actions on a semi-affine variety $X$ are in one-to-one correspondence with rational
semisimple derivations $\partial\colon\mathcal{O}_{X}\to\mathcal{O}_{X}$
such that the derivation on global sections $\partial_X\colon\OO_{X}(X)\to\OO_{X}(X)$
on the ring of global regular functions is semisimple with integers eigenvalues.
\end{proposition}

\begin{proof}
By Rosenlicht theorem \cite{Ros56}, for any regular $\GM$-action on $X$ there exists of a nonempty $\GM$-invariant affine open subset $U$. Hence, $\partial_U\colon \OO_X(U)\to \OO_X(U)$ is semisimple with integer eigenvalues and since $\OO_{X}(X)\subset\OO_{X}(U)$
it follows that $\partial_X$ is a semisimple
derivation of $\OO_{X}(X)$ with integer eigenvalues.

Conversely, let $\partial\colon\mathcal{O}_{X}\to\mathcal{O}_{X}$
be a derivation such that $\partial_{0}=\partial_X\colon\OO_{X}(X)\to \OO_{X}(X)$
is semisimple with integer eigenvalues. Then $\partial_{0}$ induces a possibly trivial
regular $\GM$-action $\varphi_{0}\colon\GM\times X_{0}\to X_{0}$
on $X_{0}=\spec \OO_X(X)$ for which the
canonical morphism $p\colon X\to X_{0}$ is $\GM$-equivariant.
In particular, for every point $x\in X$, letting $\xi=\varphi\mid_{\GM\times\left\{ x\right\} }\colon\GM\dashrightarrow X$,
$t\mapsto\varphi(t,x)$ and $\xi_{0}=\varphi_{0}\mid_{\GM\times p(x)}\colon\GM\to X_{0}$,\,$t\mapsto\varphi_{0}(t,p(x))$
, we have a commutative diagram \[\xymatrix {\GM \ar@{-->}[r]^{\xi} \ar[dr]_{\xi_0} & X \ar[d]^{p} \\ & X_0.}\]
Since $p$ is proper, we deduce from the valuative criterion for properness
applied to the local ring of every closed point $t\in\GM$
that $\varphi$ is defined at every point $\left(x,t\right)\in\GM\times X$
whence is a regular $\GM$-action on $X$. 
\end{proof}


\bibliographystyle{alpha} \bibliography{math}

\end{document}